\documentclass[11pt,reqno]{amsart}
\usepackage{graphicx}
\usepackage{verbatim}
\usepackage[usenames,dvipsnames]{color}
\newtheorem{thm}{Theorem}[section]
\newtheorem{lem}[thm]{Lemma}
\newtheorem{prop}[thm]{Proposition}
\newtheorem{cor}[thm]{Corollary}

\numberwithin{equation}{section}

\def\R{\mathbb{R}}

\def\Z{\mathbb{Z}}

\def\ra{\rightarrow}

\def\al{\alpha}

\def\ep{\epsilon}
\def\ka{\kappa}

\def\la{\lambda}

\def\om{\omega}

\def\de{\delta}

\begin{document}

\title[Transition fronts in nonlocal Fisher-KPP]{Transition fronts in nonlocal Fisher-KPP equations in time heterogeneous media}

\author{Wenxian Shen}
\address{Department of Mathematics and Statistics, Auburn University, Auburn, AL 36849}
\email{wenxish@auburn.edu}

\author{Zhongwei Shen}
\address{Department of Mathematics and Statistics, Auburn University, Auburn, AL 36849, USA}
\curraddr{Department of Mathematical and Statistical Sciences, University of Alberta, Edmonton, AB T6G 2G1, Canada}
\email{zzs0004@auburn.edu\\zhongwei@ualberta.ca}





\begin{abstract}
The present paper is devoted to the study of transition fronts of nonlocal Fisher-KPP equations in time heterogeneous media. We first construct transition fronts with exact decaying rates as the space variable tends to infinity and with prescribed interface location functions, which are natural generalizations of front location functions in homogeneous media. Then, by the general results on space regularity of transition fronts of nonlocal evolution equations proven in the authors' earlier work (\cite{ShSh14-4}), these transition fronts are continuously differentiable in space. We show that their space partial derivatives have exact decaying rates as the space variable tends to infinity. Finally, we study the asymptotic stability of transition fronts. It is shown that transition fronts attract those solutions whose initial data decays as fast as transition fronts near infinity and essentially above zero near negative infinity.
\end{abstract}

\maketitle

\tableofcontents


\section{Introduction}

In the present paper, we study transition fronts of  nonlocal dispersal  equations of the form
\begin{equation}\label{main-eqn}
u_{t}=J\ast u-u+f(t,u),\quad (t,x)\in\R\times\R,
\end{equation}
where the dispersal kernel $J$ satisfies

\begin{itemize}
\item[ \bf(H1)] {\it
$J\not\equiv0$, $J\in C^{1}(\R)$, $J(x)=J(-x)\geq0$ for all $x\in\R$, $\int_{\R}J(x)dx=1$ and
\begin{equation*}
\int_{\R}J(x)e^{\la x}dx<\infty,\quad\int_{\R}|J'(x)|e^{\la x}dx<\infty\quad\text{for any}\,\,\, \la\in\R.
\end{equation*}
}
\end{itemize}

We assume that $f$ is a Fisher-KPP or monostable type nonlinearity, that is, $f$ satisfies

\begin{itemize}
\item[\bf (H2)]
{\it
$f:\R\times[0,1]\ra\R$ is $C^{1}$ and $C^{2}$ in its first variable and second variable, respectively, and satisfies the following conditions:
\begin{itemize}
\item $f(t,0)=f(t,1)=0$ and $f(t,u)\geq0$ for all $(t,u)\in\R\times(0,1)$;
\item there exists $\theta_{1}\in(0,1)$ such that $f_{u}(t,u)\leq0$ for all $u\in[\theta_{1},1]$;
\item $\sup_{(t,u)\in\R\times[0,1]}\max\{|f_{t}(t,u)|,|f_{u}(t,u)|,|f_{uu}(t,u)|\}<\infty$;
\item there exists a $C^{1}$ function $g:[0,1]\to\R$ satisfying
\begin{itemize}
\item $g(0)=g(1)=0<g(u)$ for $u\in(0,1)$,
\item $g_{u}(0)=1$, $g_{u}(1)\geq-1$, $g_{u}$ is decreasing and $\int_{0}^{1}\frac{u-g(u)}{u^{2}}du<\infty$
\end{itemize}
such that
\begin{equation*}
a(t)g(u)\leq f(t,u)\leq a(t)u,\quad (t,u)\in\R\times[0,1],
\end{equation*}
where $a(t):=f_{u}(t,0)$ satisfies $a_{\inf}:=\inf_{t\in\R}a(t)>0$.
\end{itemize}
}
\end{itemize}

The decaying assumptions on $J$ are used implicitly in some places of the present paper. For example, to ensure Proposition \ref{prop-regularity} below, we need $\la\geq r$ (see the proof of \cite[Theorem 1.3]{ShSh14-4}). Here is a typical example of the function $f$, i.e., $f(t,u)=a(t)u(1-u)$.  Equation \eqref{main-eqn} with $J$ and $f$ satisfying (H1) and (H2) arises in, for example, population dynamics modeling the evolution of species. In which case, the unknown function $u(t,x)$ would be the normalized population density, the operator $v\mapsto J\ast v-v$ models nonlocal diffusion and $f(t,u)$ is the logistic-type growth rate. For equation \eqref{main-eqn}, solutions of particular interest are global-in-time front-like solutions, more precisely, transition fronts, since they play the key role in describing extinction and persistence of the species.

Recall that a global-in-time solution $u(t,x)$ of \eqref{main-eqn} is called a (right-moving) \textit{transition front} (connecting $0$ and $1$) in the sense of Berestycki-Hamel (see \cite{BeHa07,BeHa12}, also see \cite{Sh04, Sh11}) if $u(t,x)\in(0,1)$ for all $(t,x)\in\R\times\R$ and there is a function $X:\R\to\R$, called \textit{interface location function}, such that
\begin{equation*}
\lim_{x\to-\infty}u(t,x+X(t))=1\,\,\text{and}\,\,\lim_{x\to\infty}u(t,x+X(t))=0\,\,\text{uniformly in}\,\,t\in\R.
\end{equation*}
The interface location function $X(t)$ tells the position of the transition front as time $t$ elapses, while the uniform-in-$t$ limits shows the \textit{bounded interface width}, that is,
\begin{equation*}
\forall\,\,0<\ep_{1}\leq\ep_{2}<1,\quad\sup_{t\in\R}{\rm diam}\{x\in\R|\ep_{1}\leq u(t,x)\leq\ep_{2}\}<\infty.
\end{equation*}
Thus, transition fronts are proper generalizations of traveling waves in homogeneous media and periodic (or pulsating) traveling waves in periodic media. Notice if $\xi(t)$ is a bounded function, then $X(t)+\xi(t)$ is also an interface location function. Therefore, interface location functions are not unique. But, it is not hard to check that the difference of two interface location functions is a bounded function. Hence, interface location functions are unique up to addition by bounded functions.

We see that transition fronts can be defined in the same way for more general equations, say,
\begin{equation}\label{eqn-general}
u_{t}=J\ast u-u+f(t,x,u).
\end{equation}
Many results have been obtained for \eqref{eqn-general} when $f(t,x,u)$ is of monostable or Fisher-KPP type in various special cases. For example, when $f(t,x,u)=f(u)$, traveling waves and minimal speeds have been studied in \cite{CaCh04,CoDu05,CoDu07,Sc80}. When $f(t,x,u)=f(x,u)$ is periodic in $x$ or $f(t,x,u)$ is periodic in both $t$ and $x$, spreading properties and periodic traveling waves have been studied in \cite{CDM13,RaShZh,ShZh10,ShZh12-1,ShZh12-2}. When $f(t,x,u)=f(x,u)$, while principal eigenvalue, positive solution and long-time behavior of solutions was studied in \cite{BCV14}, transition fronts were shown to exist in \cite{LiZl14}.

In the present paper, we study transition fronts of \eqref{eqn-general} when $f(t,x,u)=f(t,u)$, that is, \eqref{main-eqn}. To state our results, we define for $\ka>0$
\vspace{-0.0in}\begin{equation*}\label{speed-function-intro}
c^{\ka}(t)=\frac{\int_{\R}J(y)e^{\ka y}dy-1}{\ka}t+\frac{1}{\ka}\int_{0}^{t}a(s)ds,\quad t\in\R.
\end{equation*}
We see that if $a(t)\equiv a>0$ is a constant function, then $c^{\ka}(t)$ is nothing but the front location function of traveling waves with speed $\frac{\int_{\R}J(y)e^{\ka y}dy-1+a}{\ka}$ of a homogeneous nonlocal Fisher-KPP equation (see e.g. \cite{CaCh04,CoDu07,Sc80}). Note that since $\inf_{t\in\R}a(t)>0$, $c^{\ka}(t)$ is increasing, and since
$\sup_{(t,u)\in\R\times[0,1]}|f_{t}(t,u)|<\infty$, $c^{\ka}(t)$ is at most linear. Indeed, if $a_{\inf}=\inf_{t\in\R}a(t)$ and $a_{\sup}=\sup_{t\in\R}a(t)$, then
\begin{equation*}
c^{\ka}(t)\in\frac{\int_{\R}J(y)e^{\ka y}dy-1}{\ka}t+\bigg[\frac{a_{\inf}}{\ka}t,\frac{a_{\sup}}{\ka}t\bigg],\quad t\in\R.
\end{equation*}
This function will serve as the interface location function as in the definition of transition fronts.

Our first result concerning the existence, regularity and decaying properties of transition fronts of \eqref{main-eqn} is stated in the following
theorem.

\begin{thm}\label{thm-tf}
Assume (H1) and (H2). There exists $\ka_{0}>0$ such that for any $\ka\in(0,\ka_{0}]$, there is a transition front $u^{\ka}(t,x)$ of \eqref{main-eqn} with interface location function $X^{\ka}(t)=c^{\ka}(t)$ and satisfying the following properties
\begin{itemize}
\item[\rm(i)] $u^{\ka}(t,x)$ is decreasing in $x$ for any $t\in\R$;

\item[\rm(ii)] there holds $\lim_{x\to\infty}\frac{u^{\ka}(t,x+X^{\ka}(t))}{e^{-\ka x}}=1$ uniformly in $t\in\R$;

\item[\rm(iii)] $u^{\ka}(t,x)$ is continuously differentiable in $x$ for any $t\in\R$ and satisfies
\begin{equation*}
\sup_{(t,x)\in\R\times\R}\frac{|u^{\ka}_{x}(t,x)|}{u^{\ka}(t,x)}<\infty;
\end{equation*}

\item[\rm(iv)] there holds $\lim_{x\to-\infty}u^{\ka}_{x}(t,x+X^{\ka}(t))=0$ uniformly in $t\in\R$;

\item[\rm(v)] there holds $\lim_{x\to\infty}\frac{u_{x}^{\ka}(t,x+X^{\ka}(t))}{u^{\ka}(t,x+X^{\ka}(t))}=-\ka$ uniformly in $t\in\R$.
\end{itemize}

\end{thm}

Note that the continuity of a transition front $u(t,x)$ in the space variable $x$ is not assumed in the definition of transition fronts. But the space regularity of transition fronts plays an important role in the study of other important properties such as stability and uniqueness of transition fronts. In the random dispersal case, the space regularity of transition fronts follows from parabolic Schauder estimates, while, thanks to the lack of space regularity for the nonlocal dispersal equations (that is, the semigroup generated by the nonlocal dispersal operator has no regularizing effect), a transition front of a nonlocal dispersal equation may not be regular in space. We refer the reader to \cite{BaFiReWa97} for the existence of discontinuous traveling waves of $u_{t}=J\ast u-u+f_{B}(u)$, where $f_{B}$ is a balanced bistable nonlinearity. In \cite{ShSh14-4}, we established some very general results on the space regularity of transition fronts of nonlocal dispersal equations. Among others, we proved in \cite{ShSh14-4} the following proposition.



\begin{prop}[\cite{ShSh14-4}]\label{prop-regularity}
Assume (H1) and (H2). Let $w(t,x)$ be an arbitrary transition front of \eqref{main-eqn} satisfying
\begin{equation}\label{harnack-type-intro}
w(t,x)\leq Ce^{r|x-y|}w(t,y),\quad (t,x,y)\in\R\times\R\times\R
\end{equation}
for some $C>0$ and $r>0$. Then, $w(t,x)$ is continuously differentiable in $x$ for any $t\in\R$ and satisfies $\sup_{(t,x)\in\R\times\R}\frac{|w_{x}(t,x)|}{w(t,x)}<\infty$.
\end{prop}

At this point,  we mention that the regularity of pulsating fronts for nonlocal KPP equations in the space periodic case was treated in \cite{CDM13} (see also \cite{ShZh12-2}). We remark that Theorem \ref{thm-tf}$\rm(iii)$ follows directly from Proposition \ref{prop-regularity}. We point out  that the existence of transition fronts in Theorem \ref{thm-tf} is proven constructively via the construction of appropriate sub- and super-solutions. The $\ka_0$ in Theorem \ref{thm-tf} is obtained in the construction of sub-solutions 
(see Proposition \ref{prop-subsol}) and satisfies that $\ka_{0}<\inf_{t\in\R}\ka_{0}(t)$ (see \eqref{an-relation-parameter-ka}), where $\ka_{0}(t)>0$ is such that 
$$
\frac{\int_{\R}J(y)e^{\ka_{0}(t) y}dy-1+a(t)}{\ka_{0}(t)}=\min_{\ka>0}\frac{\int_{\R}J(y)e^{\ka y}dy-1+a(t)}{\ka}.
$$
The $\ka_0$ in Theorem \ref{thm-tf}
 may be small and  hence the set of transition fronts obtained in Theorem \ref{thm-tf} may only contain  those which move sufficiently fast to the right. In \cite{RaShZh}, the authors  proved the existence of periodic traveling waves in the time periodic case $f(t+T,u)=f(t,u)$ also constructively via the construction of appropriate sub- and super-solutions. The $\ka_0$ obtained in \cite{RaShZh} is given by
\begin{equation}\label{k0-eq}
\frac{\int_{\R} J(y)e^{\ka_{0} y}dy-1+\hat a}{\ka_{0}}=\min_{\ka>0}\frac{\int_{\R} J(y)e^{\ka y}dy-1+\hat a}{\ka},
\end{equation}
where $\hat a=\frac{1}{T}\int_0^T a(t)dt$. The $\ka_0$ in \eqref{k0-eq} for the time periodic case is optimal and the value in \eqref{k0-eq} is nothing but the minimal speed of periodic traveling waves. But the method to construct sub-solutions in the time periodic case adopted in \cite{RaShZh} is difficult to be applied to the general time dependent case. We adopt in the present paper a method based on an idea from \cite{Zl12} (also see \cite{LiZl14,TaZhZl14}), which is different from that in \cite{RaShZh} and allows us to apply it to the general time dependent case, but does not enable us to obtain the optimal value of $\ka_0$.  It would be interesting to determine the optimal value for $\ka_{0}$ (see Subsection \ref{subsec-some-remarks} for some remarks). 

We also remark that if $a(t)=f_{u}(t,0)$ is uniquely ergodic, that is, the hull of $a(t)$ is compact and the dynamical system generated by the shift operators on the hull of $a(t)$ is uniquely ergodic, then the limit $\lim_{t\to\infty}\frac{1}{t}\int_{0}^{t}a(s)ds$ exists, and hence, the asymptotic speed
\begin{equation*}
\lim_{t\to\infty}\frac{X^{\ka}(t)}{t}=\lim_{t\to\infty}\frac{c^{\ka}(t)}{t}=\frac{\int_{\R}J(y)e^{\ka y}dy-1}{\ka}+\frac{1}{\ka}\lim_{t\to\infty}\frac{1}{t}\int_{0}^{t}a(s)ds
\end{equation*}
exists. Since interface location functions are unique up to addition by bounded functions as mentioned before, asymptotic speed (if exists) is independent of the choice of interface location functions. Note that asymptotic speed hardly exists in general.

In the presence of space regularity, i.e., Theorem \ref{thm-tf}$\rm(iii)$, we then move to the study of the asymptotic stability of $u^{\ka}(t,x)$. To do so, we further assume the uniform exponential stability of the constant solution $1$, that is,

\begin{itemize}
\item[\bf (H3)] {\it
There exists $\theta_{1}\in(0,1)$ and $\beta_{1}>0$ such that $f_{u}(t,u)\leq-\beta_{1}$ for all $(t,u)\in\R\times[\theta_{1},1]$.
}
\end{itemize}

Note that (H3)  improves the corresponding assumption in (H2). From classical semigroup theory and comparison principle, we know that for any $u_{0}\in C^{b}_{\rm unif}(\R)$, the space of real-valued, bounded and uniformly continuous functions on $\R$, the solution $u(t,x;t_{0},u_{0})$ of \eqref{main-eqn} with initial data $u(t_{0},\cdot;t_{0},u_{0})=u_{0}$ exists globally in the space $C^{b}_{\rm unif}(\R)$ and is unique. We then show

\begin{thm}\label{thm-asympt-stability}
Assume (H1)-(H3). Let $\ka_{0}$ be as in Theorem \ref{thm-tf}. For $\ka\in(0,\ka_{0}]$, let $u^{\ka}(t,x)$ be the transition front in Theorem \ref{thm-tf}. Let $u_{0}:\R\to[0,1]$ be uniformly continuous and satisfy
\begin{equation*}
\liminf_{x\to-\infty}u_{0}(x)>0\quad\text{and}\quad\lim_{x\to\infty}\frac{u_{0}(x)}{u^{\ka}(t_{0},x)}=1
\end{equation*}
for some $t_{0}\in\R$. Then, there holds the limit
\begin{equation}\label{stability-intro}
\lim_{t\to\infty}\sup_{x\in\R}\bigg|\frac{u(t,x;t_{0},u_{0})}{u^{\ka}(t,x)}-1\bigg|=0.
\end{equation}
\end{thm}


A generalization of Theorem \ref{thm-asympt-stability} is given in Corollary \ref{cor-stability-general}, where we show that if $u_{0}$ is as in Theorem \ref{thm-asympt-stability}, but with $\lim_{x\to\infty}\frac{u_{0}(x)}{u^{\ka}(t_{0},x)}=1$ replaced by $\lim_{x\to\infty}\frac{u_{0}(x)}{u^{\ka}(t_{0},x)}=\la$ for some $\la>0$, then there holds $\lim_{t\to\infty}\sup_{x\in\R}\big|\frac{u(t,x;t_{0},u_{0})}{u^{\ka}(t,x-\frac{\ln\la}{\ka})}-1\big|=0$.

More generally, if $\lim_{x\to\infty}\frac{u_{0}(x)}{e^{-\ka x}}=\tilde{\la}>0$, then, using Theorem \ref{thm-tf}$\rm(ii)$ and the facts that $X^{\ka}(t)$ is continuous, increasing and satisfies $\lim_{t\to\pm\infty}X^{\ka}(t)=\pm\infty$, there exists a unique $\tilde{t}_{0}\in\R$ such that the limit $\lim_{x\to\infty}\frac{u_{0}(x)}{u^{\ka}(\tilde{t}_{0},x)}=1$ exists, which leads to the asymptotic dynamics of $u(t,x;\tilde{t}_{0},u_{0})$ as in \eqref{stability-intro} with $t_{0}$ replaced by $\tilde{t}_{0}$. See Corollary \ref{cor-stability-more-general} for more details. In particular, if $a(t)$ is uniquely ergodic, then $u(t,x;\tilde{t}_{0},u_{0})$ has asymptotic spreading properties in the following sense: if $\lim_{t\to\infty}\frac{X^{\ka}(t)}{t}=c^{\ka}_{*}$, then for any $\ep>0$ there holds
\begin{equation*}
\lim_{t\to\infty}\inf_{x\leq (c_{*}^{\ka}-\ep)t}u(t,x;\tilde{t}_{0},u_{0})=1\quad\text{and}\quad\lim_{t\to\infty}\inf_{x\geq (c_{*}^{\ka}+\ep)t}u(t,x;\tilde{t}_{0},u_{0})=0.
\end{equation*}

We remark that results as in Theorem \ref{thm-tf} and Theorem \ref{thm-asympt-stability} can also be established for the following reaction-diffusion equation
\begin{equation}\label{eqn-rd}
u_{t}=u_{xx}+f(t,u),\quad (t,x)\in\R\times\R.
\end{equation}
In particular, we have
\begin{thm}
Assume $\rm(H2)$ and $\rm(H3)$. There exists $\ka_{0}>0$ such that for any $\ka\in(0,\ka_{0}]$, there is a transition front $u^{\ka}(t,x)$ of \eqref{eqn-rd} with interface location function
\begin{equation*}
X^{\ka}(t)=\ka t+\frac{1}{\ka}\int_{0}^{t}a(s)ds,\quad t\in\R
\end{equation*}
and satisfying the following properties
\begin{itemize}
\item[\rm(i)] $u^{\ka}_{x}(t,x)<0$ for all $(t,x)\in\R\times\R$;

\item[\rm(ii)] the limits $\lim_{x\to\infty}\frac{u^{\ka}(t,x+X^{\ka}(t))}{e^{-\ka x}}=1$ and $\lim_{x\to\infty}\frac{u_{x}^{\ka}(t,x+X^{\ka}(t))}{u^{\ka}(t,x+X^{\ka}(t))}=-\ka$ hold and are uniform in $t\in\R$;

\item[\rm(iii)] let $u_{0}:\R\to[0,1]$ be uniformly continuous and satisfy $\liminf_{x\to-\infty}u_{0}(x)>0$ and $\lim_{x\to\infty}\frac{u_{0}(x)}{e^{-\ka x}}=\la$ for some $\la>0$; then, there exists $t_{0}\in\R$ such that
\begin{equation*}
\lim_{t\to\infty}\sup_{x\in\R}\bigg|\frac{u(t,x;t_{0},u_{0})}{u^{\ka}(t,x)}-1\bigg|=0.
\end{equation*}
\end{itemize}
\end{thm}

It should be pointed out that transition fronts of \eqref{eqn-rd} were established in \cite{NaRo12}, while no result concerning the stability exists in the literature. In the case $f(t,u)$ being uniquely ergodic in $t$, the existence, stability and uniqueness of transition fronts of \eqref{eqn-rd} were studied in \cite{Sh11}.

Finally, we remark that while transition fronts of reaction-diffusion equations and nonlocal equations of Fisher-KPP type in general time or space heterogeneous media have been studied (see e.g. \cite{LiZl14,NaRo12,NRRZ12,TaZhZl14,Zl12}), there exists no result in the literature concerning the corresponding discrete equations in general time or space heterogeneous media, i.e.,
\begin{equation}\label{discrete-time}
\dot{u}_{i}=u_{i+1}-2u_{i}+u_{i-1}+f(t,u_{i}),\quad t\in\R,\,\,i\in\Z,
\end{equation}
or
\begin{equation}\label{discrete-space}
\dot{u}_{i}=u_{i+1}-2u_{i}+u_{i-1}+f(i,u_{i}),\quad t\in\R,\,\,i\in\Z,
\end{equation}
where $f(t,u)$ and $f(i,u)$ are of Fisher-KPP type. Such discrete equations also arises naturally in applications (see e.g. \cite{ShSw90}), and hence, it is of great importance to study them. We refer the reader to \cite{CFG06,ChGu02,ChGu03,FGS02,HuZi94,WuZo97,ZHH93} and references therein for works of \eqref{discrete-time} or \eqref{discrete-space} in homogeneous media, i.e., $f(t,u)=f(u)$ or $f(i,u)=f(u)$, and to \cite{GuHa06,GuWu09} for works of \eqref{discrete-space} in periodic media, i.e., $f(i,u)=f(i+L,u)$ for some $L\in\Z$.

The rest of the paper is organized as follows. In Section \ref{sec-sub-super-sol}, we construct appropriate global-in-time sub- and super-solutions of \eqref{main-eqn} for the use to construct transition fronts. In Section \ref{sec-construct-tf}, we construct transition fronts and prove Theorem \ref{thm-tf}. In Section \ref{sec-stability}, we study the stability of transition fronts constructed in Theorem \ref{thm-tf} and prove Theorem \ref{thm-asympt-stability}.


\section{Construction of sub- and super-solutions}\label{sec-sub-super-sol}

In this section, we construct appropriate global-in-time sub- and super-solutions of the equation \eqref{main-eqn}. Throughout this section, we assume (H1) and (H2).

\subsection{Construction of the super-solution}

For $\ka>0$, let
\begin{equation}\label{speed-function}
c^{\ka}(t)=\frac{\int_{\R}J(y)e^{\ka y}dy-1}{\ka}t+\frac{1}{\ka}\int_{0}^{t}a(s)ds,\quad t\in\R.
\end{equation}

For $\ka>0$, define
\begin{equation}\label{supsol}
\phi^{\ka}(t,x)=e^{-\ka(x-c^{\ka}(t))},\quad (t,x)\in\R\times\R.
\end{equation}

\begin{prop}\label{prop-subsol}
For any $\ka>0$, $\phi^{\ka}(t,x)$ is a super-solution of \eqref{main-eqn}.
\end{prop}
\begin{proof}
We check that $\phi^{\ka}(t,x)$ solves $\phi_{t}=J\ast\phi-\phi+a(t)\phi$. The proposition then follows from $f(t,u)\leq a(t)u$ for $u\geq0$ by (H2) (note it is safe to extend $f(t,u)$ to $u\in(1,\infty)$ in this way).
\end{proof}


\subsection{Construction of the sub-solution}

To construct the sub-solution, we borrow an idea from \cite{Zl12} (also see \cite{LiZl14,TaZhZl14}). Let us consider the homogeneous reaction-diffusion equation
\begin{equation}\label{eqn-homo-g}
u_{t}=u_{xx}+g(u),
\end{equation}
where $g$, given in (H2), is of Fisher-KPP type with $g_{u}(0)=1$. It is known (see e.g. \cite{Uch78}) that for any $\al\in(0,1)$, traveling wave of the form $\phi_{\al}(x-c_{\al}t)$ with $\phi_{\al}(-\infty)=1$ and $\phi_{\al}(\infty)=0$ exists and is unique up to space translation, where $c_{\al}=\al+\frac{1}{\al}$. We assume, without loss of generality, that $\lim_{x\to\infty}e^{\al x}\phi_{\al}(x)=1$. Moreover, for any $\al\in(0,1)$, the exponential function $e^{-\al(x-c_{\al}t)}$ is a traveling wave of the linearization of \eqref{eqn-homo-g} at $0$, that is,
\begin{equation}\label{eqn-homo-g-linear}
u_{t}=u_{xx}+u.
\end{equation}
For $\al\in(0,1)$, we define a function $h_{\al}:[0,\infty)\to[0,1)$ by setting
\begin{equation*}
h_{\al}(s)=\begin{cases}
0,\quad&s=0,\\
\phi_{\al}(-\al^{-1}\ln s),\quad&s>0.
\end{cases}
\end{equation*}
Note that $h_{\al}(e^{-\al x})=\phi_{\al}(x)$, that is, $h_{\al}$ takes the profile of traveling waves of \eqref{eqn-homo-g-linear} with speed $c_{\al}$ to that of \eqref{eqn-homo-g} with the same speed. Since $\phi_{\al}$ solves $\phi_{\al}''+c_{\al}\phi_{\al}'+g(\phi_{\al})=0$, we check that $h_{\al}$ satisfies
\begin{equation}\label{eqn-h}
\al^{2}s^{2}h_{\al}''(s)-sh_{\al}'(s)+g(h_{\al}(s))=0.
\end{equation}

\begin{lem}\label{the-function-h}
Let $\al\in(0,1)$. Then,
\begin{itemize}
\vspace{-0,05in}\item[\rm(i)] $h_{\al}$ is increasing and satisfies $\lim_{s\to\infty}h_{\al}(s)=1$, $h_{\al}'(0)=1$ and $h_{\al}''(s)<0$ for $s>0$;
\item[\rm(ii)] let $\beta=2+\frac{1}{\al^{2}}$ and $\rho_{\al}(x)=-h_{\al}''(e^{-x})$ for $x\in\R$; then there holds
\begin{equation*}
0<\rho_{\al}(x)\leq e^{\beta|x-y|}\rho_{\al}(y),\quad x,y\in\R.
\end{equation*}
\end{itemize}
\end{lem}
\begin{proof}
See \cite{Zl12} for $\rm(i)$ and \cite[Lemma 3.1]{LiZl14} for $\rm(ii)$.
\end{proof}

Set $\al=\frac{3}{4}$ (it is nothing special). We write $h(s):=h_{\frac{3}{4}}(s)$ and $\rho(x):=\rho_{\frac{3}{4}}(x)$. For $\ka>0$, define
\begin{equation*}
\psi^{\ka}(t,x)=h(\phi^{\ka}(t,x)),\quad (t,x)\in\R\times\R,
\end{equation*}
where $\phi^{\ka}(t,x)=e^{-\ka(x-c^{\ka}(t))}$ is given in \eqref{supsol}. Then,

\begin{prop}\label{prop-subsol}
There exists $\ka_{0}>0$ such that for any $\ka\in(0,\ka_{0}]$, $\psi^{\ka}(t,x)$ is a sub-solution of $u_{t}=J\ast u-u+a(t)g(u)$; in particular, it is a sub-solution of \eqref{main-eqn} .
\end{prop}
\begin{proof}
Since $\phi^{\ka}(t,x)$ satisfies the equation $\phi^{\ka}_{t}=J\ast\phi^{\ka}-\phi^{\ka}+a(t)\phi^{\ka}$, we readily check that $N[\psi^{\ka}]:=\psi^{\ka}_{t}-[J\ast\psi^{\ka}-\psi^{\ka}]$ satisfies
\begin{equation*}
N[\psi^{\ka}]=a(t)\phi^{\ka}h'(\phi^{\ka})+h'(\phi^{\ka})[J\ast\phi^{\ka}-\phi^{\ka}]-\int_{\R}J(y)[h(\phi^{\ka}(t,x-y))-h(\phi^{\ka}(t,x))]dy.
\end{equation*}
By the second-order Taylor expansion, i.e.,
\begin{equation*}
\begin{split}
h(\phi^{\ka}(t,x-y))-h(\phi^{\ka}(t,x))&=h'(\phi^{\ka}(t,x))[\phi^{\ka}(t,x-y)-\phi^{\ka}(t,x)]\\
&\quad+\frac{h''(\zeta_{x,y,t})}{2}[\phi^{\ka}(t,x-y)-\phi^{\ka}(t,x)]^{2}
\end{split}
\end{equation*}
for some $\zeta_{x,y,t}$ between $\phi^{\ka}(t,x-y)$ and $\phi^{\ka}(t,x)$, we see
\begin{equation*}
N[\psi^{\ka}]=a(t)\phi^{\ka}h'(\phi^{\ka})-\frac{1}{2}\int_{\R}J(y)h''(\zeta_{x,y,t})[\phi^{\ka}(t,x-y)-\phi^{\ka}(t,x)]^{2}dy.
\end{equation*}
For the integral in the above equality, we first see from the monotonicity of $\ln$ and the explicit expression of $\phi^{\ka}(t,x)$ that
\begin{equation*}
|\ln\zeta_{x,y,t}-\ln\phi^{\ka}(t,x)|\leq|\ln\phi^{\ka}(t,x-y)-\ln\phi^{\ka}(t,x)|\leq\ka|y|.
\end{equation*}
It then follows from Lemma \ref{the-function-h}$\rm(ii)$ with $\beta=2+\frac{1}{(3/4)^{2}}<4$ that
\begin{equation*}
-h''(\zeta_{x,y,t})=\rho(-\ln\zeta_{x,y,t})\leq e^{4|\ln\zeta_{x,y,t}-\ln\phi^{\ka}(t,x)|}\rho(-\ln\phi^{\ka}(t,x))\leq e^{4\ka|y|}[-h''(\phi^{\ka}(t,x))].
\end{equation*}
Hence,
\begin{equation*}
\begin{split}
&-\frac{1}{2}\int_{\R}J(y)h''(\zeta_{x,y,t})[\phi^{\ka}(t,x-y)-\phi^{\ka}(t,x)]^{2}dy\\
&\quad\quad\leq-\frac{1}{2}h''(\phi^{\ka}(t,x))\int_{\R}J(y)e^{4\ka|y|}[\phi^{\ka}(t,x-y)-\phi^{\ka}(t,x)]^{2}dy\\
&\quad\quad=-\frac{1}{2}h''(\phi^{\ka}(t,x))[\phi^{\ka}(t,x)]^{2}\int_{\R}J(y)e^{4\ka|y|}[e^{\ka y}-1]^{2}dy,
\end{split}
\end{equation*}
which leads to
\begin{equation}\label{an-equality-678543}
N[\psi^{\ka}]\leq a(t)\phi^{\ka}(t,x)h'(\phi^{\ka}(t,x))-\frac{1}{2}h''(\phi^{\ka}(t,x))[\phi^{\ka}(t,x)]^{2}\int_{\R}J(y)e^{4\ka|y|}[e^{\ka y}-1]^{2}dy.
\end{equation}
Since $\frac{d}{d\ka}[e^{\ka y}-1]^{2}=2(e^{\ka y}-1)e^{\ka y}y\geq0$, the function $\ka\mapsto\int_{\R}J(y)e^{4\ka|y|}[e^{\ka y}-1]^{2}dy$ is increasing on $(0,\infty)$. Moreover, we have
\begin{equation*}
\int_{\R}J(y)e^{4\ka|y|}[e^{\ka y}-1]^{2}dy\to\begin{cases}
0,\quad& \ka\to0,\\
\infty,\quad&\ka\to\infty.
\end{cases}
\end{equation*}
Thus, due to $\inf_{t\in\R}a(t)>0$ by (H2),  there is a unique $\ka_{0}>0$ such that
\begin{equation}\label{parameter-ka-bound}
\frac{1}{2}\int_{\R}J(y)e^{4\ka_{0}|y|}[e^{\ka_{0} y}-1]^{2}dy=\bigg(\frac{3}{4}\bigg)^{2}\inf_{t\in\R}a(t),
\end{equation}
which implies $\frac{1}{2}\int_{\R}J(y)e^{4\ka|y|}[e^{\ka y}-1]^{2}dy\leq\big(\frac{3}{4}\big)^{2}a(t)$ for $t\in\R$ and $\ka\in(0,\ka_{0}]$. It then follows from \eqref{an-equality-678543} that
\begin{equation*}
N[\psi^{\ka}]\leq a(t)\bigg[\phi^{\ka}(t,x)h'(\phi^{\ka}(t,x))-\bigg(\frac{3}{4}\bigg)^{2}h''(\phi^{\ka}(t,x))[\phi^{\ka}(t,x)]^{2}\bigg]=a(t)g(\psi^{\ka}(t,x)),
\end{equation*}
where we used \eqref{eqn-h} with $s=\phi^{\ka}(t,x)$ and $\al=\frac{3}{4}$. The proposition then follows from $a(t)g(u)\leq f(t,u)$ for $u\in[0,1]$ by (H2).
\end{proof}


\subsection{Some remarks}\label{subsec-some-remarks}

We make some remarks about the sub- and super-solutions constructed in the above two subsections.

\smallskip

\noindent (i) From Lemma \ref{the-function-h}$\rm(i)$, there holds $h_{\al}(s)\leq s$ for all $s\geq0$. This, in particular, implies that $\psi^{\ka}(t,x)=h(\phi^{\ka}(t,x))\leq\phi^{\ka}(t,x)$ for all $(t,x)\in\R\times\R$ (actually, the strict inequality holds). This order relation between $\psi^{\ka}(t,x)$ and $\phi^{\ka}(t,x)$ is important in establishing various properties of approximating solutions, which will be studied in the next section, Section \ref{sec-construct-tf}. Moreover, $\psi^{\ka}(t,x)$ and $\phi^{\ka}(t,x)$ propagate to the right with the same speed $\dot{c}^{\ka}(t)$. This fact says that any global-in-time solution of \eqref{main-eqn} between $\psi^{\ka}(t,x)$ and $\min\{1,\phi^{\ka}(t,x)\}$ is a transition front.

\smallskip

\noindent (ii) In constructing the sub-solution $\psi^{\ka}(t,x)$, we restrict $\ka$ to take values in $(0,\ka_{0}]$, where $\ka_{0}>0$ is given by \eqref{parameter-ka-bound}. If we let $\ka_{0}(t)>0$ be the one minimizing $\dot{c}^{\ka}(t)$, that is,
\begin{equation}\label{minimizer-148}
\dot{c}^{\ka_{0}(t)}(t)=\min_{\ka>0}\dot{c}^{\ka}(t).
\end{equation}
then we have
\begin{equation}\label{an-relation-parameter-ka}
\ka_{0}<\inf_{t\in\R}\ka_{0}(t).
\end{equation}
In fact, notice the critical points for $\dot{c}^{\ka}(t)$, i.e., points satisfying $\frac{d}{d\ka}\dot{c}^{\ka}(t)=0$, satisfies
\begin{equation*}
\ka\int_{\R}J(y)ye^{\ka y}dy-\int_{\R}J(y)e^{\ka y}dy+1=a(t).
\end{equation*}
Setting
\begin{equation*}
q_{1}(\ka)=\ka\int_{\R}J(y)ye^{\ka y}dy-\int_{\R}J(y)e^{\ka y}dy+1,
\end{equation*}
we see that $q_{1}(0)=0$ and $q_{1}'(\ka)=\ka\int_{\R}J(y)y^{2}e^{\ka y}dy>0$ for $\ka>0$. Setting
\begin{equation*}
q_{2}(\ka)=\frac{1}{2}\bigg(\frac{4}{3}\bigg)^{2}\int_{\R}J(y)e^{4\ka|y|}[e^{\ka y}-1]^{2}dy,
\end{equation*}
we see that $q_{2}(0)=0$ and
\begin{equation*}
\begin{split}
q_{2}'(\ka)&=\bigg(\frac{4}{3}\bigg)^{2}\int_{\R}J(y)2|y|e^{4\ka|y|}[e^{\ka y}-1]^{2}dy+\bigg(\frac{4}{3}\bigg)^{2}\int_{\R}J(y)e^{4\ka|y|}[e^{\ka y}-1]e^{\ka y}ydy\\
&\geq\int_{\R}J(y)e^{4\ka|y|}[e^{\ka y}-1]e^{\ka y}ydy=\int_{\R}J(y)e^{4\ka|y|}e^{\ka y_{*}}\ka ye^{\ka y}ydy\quad(\text{where}\,\,y_{*}\in[0,y])\\
&=\ka\int_{\R}J(y)e^{2\ka|y|}y^{2}e^{\ka(|y|+y_{*})}e^{\ka(|y|+y)}dy\geq\ka\int_{\R}J(y)e^{2\ka|y|}y^{2}dy>q_{1}'(\ka),\quad\ka>0.
\end{split}
\end{equation*}
This simply means that $q_{2}(\ka)>q_{1}(\ka)$ for $\ka>0$. Hence, $q_{1}(\ka_{0})<q_{2}(\ka_{0})=\inf_{t\in\R}a(t)$. Since $q_{1}(\ka)$ is continuous and increasing, we conclude \eqref{an-relation-parameter-ka}. We will use \eqref{an-relation-parameter-ka} in the proof of Lemma \ref{lem-asymptotic-at-infty}, which is the key to the stability of transition fronts.

\medskip

\noindent (iii) A possible way to enlarge the value of $\ka_{0}$ is to make $\al$ change (we have fixed $\al=\frac{3}{4}$ in the above analysis). But, it seems not enough to push $\ka_{0}$ arbitrary close to $\inf_{t\in\R}\ka_{0}(t)$. Also, $\inf_{t\in\R}\ka_{0}(t)$ is hardly the optimal value for $\ka_{0}$, since in the periodic case, $\ka_0$ is exactly such that
$$
\frac{\int_{\R}J(y)e^{\ka_0 y}dy -1+\hat a}{\ka_0}=\min_{\ka >0} \frac{\int_{\R}J(y)e^{\ka y}dy -1+\hat a}{\ka}
$$
where $\hat a=\frac{1}{T}\int_0^T a(t)dt$ and $T$ is the period (see e.g. \cite{RaShZh}).


\section{Construction of transition fronts}\label{sec-construct-tf}

In this section, we construct transition fronts and study their space regularity and decaying properties, that is, we are going to prove Theorem \ref{thm-tf}. Throughout this section, we assume (H1) and (H2).

Fix any $\ka\in(0,\ka_{0}]$, where $\ka_{0}>0$ is given in Proposition \ref{prop-subsol}. For this fixed $\ka$, we write $c(t)=c^{\ka}(t)$, $\phi(t,x)=\phi^{\ka}(t,x)$ and $\psi(t,x)=\psi^{\ka}(t,x)$. Again, $h(s)=h_{\frac{3}{4}}(s)$.

The proof of the existence of transition fronts is constructive. We first construct approximating solutions. For $n\geq1$, let $u^{n}(t,x)$, $t\geq-n$ be the unique solution of
\begin{equation}\label{eqn-approximating-sol}
\begin{cases}
u_{t}=J\ast u-u+f(t,u),\quad t>-n,\\
u(-n,x)=\psi(-n,x)=h(\phi(-n,x)).
\end{cases}
\end{equation}
We prove some basic properties of $u^{n}(t,x)$ in the following

\begin{lem}\label{lem-harnack-type}
The following statements hold:
\begin{itemize}
\item[\rm(i)] for any $(t,x)\in[-n,\infty)\times\R$ and $n\geq1$, there holds
\begin{equation*}
0<\psi(t,x)\leq u^{n}(t,x)\leq\tilde{\phi}(t,x):=\min\{1,\phi(t,x)\};
\end{equation*}

\item[\rm(ii)] for any $(t,x,y)\in[-n,\infty)\times\R\times\R$ and $n\geq1$, there holds
\begin{equation*}
u^{n}(t,x)\leq\frac{e^{\ka|x-y|}}{h(1)}u^{n}(t,y);
\end{equation*}

\item[\rm(iii)] if $n>m\geq1$, then $u^n(t,x)\ge u^m(t,x)$ for all $(t,x)\in[-m,\infty)\times\R$.
\end{itemize}
\end{lem}
\begin{proof}
$\rm(i)$ By Lemma \ref{the-function-h}$\rm(i)$, we have $h(s)\leq s$ for all $s\geq0$, which implies that $u(-n,x)\leq\phi(-n,x)$. The result then follows from comparison principle.

$\rm(ii)$ Using the expression $\phi(t,x)=e^{-\ka(x-c(t))}$, we readily check $\tilde{\phi}(t,x)\leq e^{\ka|x-y|}\tilde{\phi}(t,y)$. Moreover, since $h(0)=0$ and $h''(s)<0$ for all $s>0$, we find $h(1)s\leq h(s)$ for $s\in[0,1]$. In particular, by the monotonicity of $h$, we have
\begin{equation*}
h(1)\tilde{\phi}(t,x)\leq h(\tilde{\phi}(t,x))\leq h(\phi(t,x))=\psi(t,x)\leq u^{n}(t,x),
\end{equation*}
which implies that
\begin{equation*}
u^{n}(t,x)\leq\tilde{\phi}(t,x)\leq e^{\ka|x-y|}\tilde{\phi}(t,y)\leq\frac{e^{\ka|x-y|}}{h(1)}u^{n}(t,y).
\end{equation*}

$\rm(iii)$ Let $n>m\geq1$. Since $\psi(t,x)$ is a sub-solution of \eqref{main-eqn} and $\psi(-n,\cdot)=u^{n}(-n,\cdot)$, comparison principle yields $\psi(t,\cdot)\leq u^{n}(t,\cdot)$ for all $t\geq-n$. In particular, at time $t=-m$, we have $u^{n}(-m,\cdot)\geq\psi(-m,\cdot)=u^{m}(-m,\cdot)$. Again, by comparison principle, we arrive at the result.
\end{proof}

The sequence $\{u^{n}(t,x)\}$ is the approximating solutions. However, we can not conclude immediately the convergence of $\{u^{n}(t,x)\}$ to a global-in-time solution of \eqref{main-eqn} due to the lack of regularity. Following the arguments in \cite{LiZl14}, we can show the uniform Lipschitz continuity of $\{u^{n}(t,x)\}$ in $x$, which of course ensures the convergence. Here, we take a different approach, which is based on the monotonicity of $\{u^{n}(t,x)\}$ as in Lemma \ref{lem-harnack-type}$\rm(iii)$.

Now, we prove Theorem \ref{thm-tf}.

\begin{proof}[Proof of Theorem \ref{thm-tf}]
We first construct transition fronts. By Lemma \ref{lem-harnack-type}$\rm(iii)$, for any fixed $(t,x)\in\R\times\R$, there exists $n_{0}=n_{0}(t)\geq1$ such that the sequence $\{u^{n}(t,x)\}_{n\geq n_{0}}$ is nondecreasing. Since it is clearly between $0$ and $1$, the limit $\lim_{n\to\infty}u^{n}(t,x)$ exists and equals to some number in $[0,1]$. Hence, there exists a function $u:\R\times\R\to[0,1]$ such that
\begin{equation}\label{locally-uniform-convergence-to-tf}
\lim_{n\to\infty}u^{n}(t,x)=u(t,x)\quad\text{pointwise in}\quad(t,x)\in\R\times\R.
\end{equation}
Moreover, since $u^{n}(t,x)$ satisfies $u^{n}_{t}=J\ast u^{n}-u^{n}+f(t,u^{n})$,
we have that for any $t>-n$ and $x\in\R$,
\begin{equation}\label{existence-eq1}
u^n(t,x)=u^n(0,x)+\int_0^t \int_{\R}J(y-x)u^n(s,y)dyds-\int_0^t u^n(s,x)ds+\int_0^t f(s,u^n(s,x))ds.
\end{equation}
Passing to the limit $n\to\infty$ in \eqref{existence-eq1}, we conclude from the dominated convergence theorem that for any $t\in\R$ and $x\in\R$,
\begin{equation*}\label{existence-eq2}
u(t,x)=u(0,x)+\int_0^t \int_{\R}J(y-x)u(s,y)dyds-\int_0^t u(s,x)ds+\int_0^t f(s,u(s,x))ds.
\end{equation*}
From which, we conclude that $u(t,x)$ is differentiable in $t$ and  satisfies \eqref{main-eqn}.

To see $u(t,x)$ is a transition front of \eqref{main-eqn}, we notice $\psi(t,x)\leq u(t,x)\leq\phi(t,x)$ by Lemma \ref{lem-harnack-type}$\rm(i)$ and \eqref{locally-uniform-convergence-to-tf}. Taking $X(t)=c(t)$, we find
\begin{equation}\label{expo-decaying-tf}
h(e^{-\ka x})=\psi(t,x+X(t))\leq u(t,x+X(t))\leq\phi(t,x+X(t))=e^{-\ka x},
\end{equation}
which implies that $u(t,x)$ is a transition front of \eqref{main-eqn}.

We now prove the properties $\rm(i)$-$\rm(v)$ listed in the statement of Theorem \ref{thm-tf}.

$\rm(i)$ Since $\{u^{n}(t,x)\}_{n\geq1}$ are decreasing in $x$, $u(t,x)$ is nonincreasing in $x$. Since $u(t,x)$ is global-in-time, it is decreasing in $x$ due to comparison principle (see e.g. \cite[Proposition A.1(iii)]{ShSh14}).

$\rm(ii)$ By \eqref{expo-decaying-tf}, we have $\frac{h(e^{-\ka x})}{e^{-\ka x}}\leq\frac{u(t,x+X(t))}{e^{-\ka x}}\leq1$. Since $\lim_{x\to\infty}\frac{h(e^{-\ka x})}{e^{-\ka x}}=h'(0)=1$ due to Lemma \ref{the-function-h}$\rm(i)$, we find that the limit $\lim_{x\to\infty}\frac{u(t,x+X(t))}{e^{-\ka x}}=1$ exists and is uniform in $t\in\R$. We also have
\begin{equation}\label{asymptotic-9564}
\lim_{x\to\infty}\frac{u(t,x+X(t))}{h(e^{-\ka x})}=1\quad\text{uniformly in}\quad t\in\R.
\end{equation}

$\rm(iii)$ By \eqref{locally-uniform-convergence-to-tf} and Lemma \ref{lem-harnack-type}$\rm(ii)$, we have
\begin{equation}\label{harnack-type-tf}
u(t,x)\leq\frac{e^{\ka|x-y|}}{h(1)}u(t,y), \quad(t,x,y)\in\R\times\R\times\R.
\end{equation}
This verifies \eqref{harnack-type-intro}. The result then follows from Proposition \ref{prop-regularity}.

$\rm(iv)$ By $\rm(iii)$, $v(t,x):=u_{x}(t,x)$ exists and satisfies the equation
\begin{equation*}
v_{t}=a(t,x)v+b(t,x),
\end{equation*}
where
\begin{equation*}
a(t,x)=-1+f_{u}(t,u(t,x))\quad\text{and}\quad b(t,x)=\int_{\R}J'(y)u(t,x-y)dy.
\end{equation*}
We see that for any $t\in\R$ and $T>0$, there holds
\begin{equation}\label{aux-eqn-eqn}
v(t,x)=e^{\int_{t-T}^{t}a(s,x)ds}v(t-T,x)+\int_{t-T}^{t}e^{\int_{\tau}^{t}a(s,x)ds}b(\tau,x)d\tau.
\end{equation}
We are going to show $v(t,x+X(t))\to0$ as $x\to-\infty$ uniformly in $t\in\R$. To do so, we let $\ep>0$, and choose $T=T(\ep)>0$ and $L=L(\ep)>0$ such that
\begin{equation*}
Ce^{-T}\leq\frac{\ep}{3}\quad\text{and}\quad\int_{\R\backslash[-L,L]}|J'(x)|dx\leq\frac{\ep}{3},
\end{equation*}
where $C:=\sup_{(t,x)\in\R\times\R}|u_{x}(t,x)|<\infty$ by $\rm(iv)$. Notice such a $L$ exists due to (H1). For $\theta_{1}$ as in (H2), let $X_{\theta_{1}}(t)$ be the interface location function at $\theta_{1}$, i.e., $u(t,X_{\theta_{1}}(t))=\theta_{1}$ for all $t\in\R$. It is well-defined due to the monotonicity of $u(t,x)$ in $x$. Since $u(t,x)$ is a transition front, we have $\sup_{t\in\R}|X_{\theta_{1}}(t)-X(t)|<\infty$. Setting
\begin{equation*}
C_{T}=\Big[\sup_{t\in\R}\dot{X}(t)\Big]T+\sup_{t\in\R}|X_{\theta_{1}}(t)-X(t)|,
\end{equation*}
we readily check that if $x-X(t)\leq-C_{T}$, then $u(s,x)\geq\theta_{1}$ for all $s\in[t-T,t]$. As a result,
\begin{equation*}
a(s,x)=-1+f_{u}(s,u(s,x))\leq-1,\quad s\in[t-T,t],\quad x-X(t)\leq-C_{T}.
\end{equation*}
It then follows from \eqref{aux-eqn-eqn} that for $x-X(t)\leq-C_{T}$ there holds
\begin{equation}\label{aux-eqn-eqn-1}
|v(t,x)|\leq Ce^{-T}+\int_{t-T}^{t}e^{-(t-\tau)}|b(\tau,x)|d\tau\leq\frac{\ep}{3}+\int_{t-T}^{t}e^{-(t-\tau)}|b(\tau,x)|d\tau.
\end{equation}
For $b(\tau,x)$, we have
\begin{equation}\label{aux-eqn-eqn-2}
\begin{split}
|b(\tau,x)|&\leq\int_{\R\backslash[-L,L]}|J'(x)|dx+\bigg|\int_{-L}^{L}J'(y)u(\tau,x-y)dy\bigg|\\
&\leq\frac{\ep}{3}+\bigg|\int_{-L}^{L}J'(y)u(\tau,x-y)dy\bigg|.
\end{split}
\end{equation}
Since $J'$ is odd and $u(\tau,x+X(\tau))\to1$ as $x\to-\infty$ uniformly in $\tau\in\R$, we find some $C_{T,L}\geq C_{T}$ such that if $x-X(t)\leq-C_{T,L}$, then
\begin{equation*}
\bigg|\int_{-L}^{L}J'(y)u(\tau,x-y)dy\bigg|\leq\frac{\ep}{3},\quad\tau\in[t-T,t].
\end{equation*}
We then deduce from \eqref{aux-eqn-eqn-1} and \eqref{aux-eqn-eqn-2} that for $x-X(t)\leq-C_{T,L}$ there holds
\begin{equation*}
|v(t,x)|\leq\frac{\ep}{3}+\frac{2\ep}{3}\int_{t-T}^{t}e^{-(t-\tau)}d\tau\leq\ep.
\end{equation*}
Observing the above analysis is uniform in $t\in\R$, we find the limit.

$\rm(v)$ By $\rm(ii)$, we have $\lim_{x\to\infty}\frac{u(t,x+X(t))}{e^{-\ka x}}=1$ uniformly in $t\in\R$. It follows that for any $\ep\in(0,1]$, there exists $M(\ep)>0$ such that
\begin{equation*}
u(t,x+X(t))\geq(1-\ep^{2})e^{-\ka x},\quad x\geq M(\ep),\quad t\in\R.
\end{equation*}
By Taylor expansion and \eqref{expo-decaying-tf}, we find for $x\geq M(\ep)$ and $t\in\R$, there exists $\ep_{*}=\ep_{*}(t,x)\in(0,\ep)$ such that
\begin{equation*}
\begin{split}
u_{x}(t,x+\ep_{*}+X(t))&=\frac{u(t,x+\ep+X(t))-u(t,x+X(t))}{\ep}\\
&\leq\frac{e^{-\ka(x+\ep)}-e^{-\ka x}+\ep^{2}e^{-\ka x}}{\ep}=e^{-\ka x}\bigg(\frac{e^{-\ka\ep}-1}{\ep}+\ep\bigg),
\end{split}
\end{equation*}
which leads to $\frac{u_{x}(t,x+\ep_{*}+X(t))}{e^{-\ka(x+\ep_{*})}}\leq e^{\ka\ep}\big(\frac{e^{-\ka\ep}-1}{\ep}+\ep\big)$. In particular, for all $x\geq M(\ep)+1$ and $t\in\R$,
\begin{equation}\label{estimate-one-side-1}
\frac{u_{x}(t,x+X(t))}{e^{-\ka x}}\leq e^{\ka\ep}\bigg(\frac{e^{-\ka\ep}-1}{\ep}+\ep\bigg).
\end{equation}

Similarly, using $\lim_{x\to\infty}\frac{u(t,x+X(t))}{h(e^{-\ka x})}=1$ uniformly in $t\in\R$ by \eqref{asymptotic-9564}, we can show that for any $\ep\in(0,1]$, there exists $\tilde{M}(\ep)>0$ such that for $x\geq\tilde{M}(\ep)+1$ and $t\in\R$ there holds
\begin{equation}\label{estimate-one-side-2}
\frac{u_{x}(t,x+X(t))}{e^{-\ka x}}\geq(1+\ep)\frac{e^{-\ka \ep}-1}{\ep}-\ep(1+\ep).
\end{equation}

Since $\lim_{\ep\to0}\frac{e^{-\ka \ep}-1}{\ep}=-\ka$, \eqref{estimate-one-side-1} and \eqref{estimate-one-side-2} imply that the limit $\lim_{x\to\infty}\frac{u_{x}(t,x+X(t))}{e^{-\ka x}}=-\ka$ holds and is uniform in $t\in\R$. The result then follows from $\rm(ii)$.
\end{proof}


\section{Stability of transition fronts}\label{sec-stability}

Let $\ka_{0}$ be as in Proposition \ref{prop-subsol} and fix $\ka\in(0,\ka_{0}]$. Let $u(t,x)=u^{\ka}(t,x)$ be the transition front constructed in Theorem \ref{thm-tf} with interface location function $X(t)=c^{\ka}(t)$. We study the asymptotic stability of $u(t,x)$. Throughout this section, we assume (H1)-(H3).

\subsection{Sub- and super-solutions}

In this subsection, we construct appropriate sub- and super-solutions. We prove

\vspace{-0.04in}\begin{prop}\label{prop-sub-sup-solution}
Let $\ep_{*}\in(0,1-\theta_{1})$ and $\om_{*}=\beta_{1}$, where $\theta_{1}\in(0,1)$ and $\beta_{1}>0$ are as in (H3). Then, for any $\ep\in(0,\ep_{*}]$ and $\om\in(0,\om_{*}]$, there exists $C=C(\ep_{*},\om)>0$ such that for any $\xi\in\R$
\begin{equation*}
\begin{split}
u^{-}(t,x)&=(1-\ep e^{-\om(t-\tau)})u(t,x+\xi-C\ep e^{-\om(t-\tau)}),\quad t\geq\tau\quad\text{and}\\
u^{+}(t,x)&=(1+\ep e^{-\om(t-\tau)})u(t,x+\xi+C\ep e^{-\om(t-\tau)}),\quad t\geq\tau
\end{split}
\end{equation*}
are sub-solution and sup-solution of \eqref{main-eqn}, respectively.
\end{prop}

To prove Proposition \ref{prop-sub-sup-solution}, we need the uniform steepness of $u(t,x)$ given in the following

\begin{lem}\label{lem-uniform-steepness}
For any $M>0$, there holds $\sup_{t\in\R}\sup_{|x-X(t)|\leq M}u_{x}(t,x)<0$.
\end{lem}
\begin{proof}
It can be proven along the same line as that of \cite[Theorem 3.1]{ShSh14-3}. So we omit it.
\end{proof}

Now, we prove Proposition \ref{prop-sub-sup-solution}

\begin{proof}[Proof of Proposition \ref{prop-sub-sup-solution}]
Here, we only show that $u^{-}(t,x)$ is a sub-solution; $u^{+}(t,x)$ being a sup-solution can be proven along the same line.

Note that by the space homogeneity of the equation \eqref{main-eqn}, we may assume, without loss of generality, that $\xi=0$. Hence, we assume
\begin{equation*}
u^{-}=u^{-}(t,x)=(1-\ep e^{-\om(t-\tau)})u(t,x-C\ep e^{-\om(t-\tau)}),\quad t\geq\tau,
\end{equation*}
where $\ep\in(0,\ep_{*}]$, $\om\in(0,\om_{*}]$ and $C>0$ (to be chosen).

With $N[u^{-}]:=u^{-}_{t}-[J\ast u^{-}-u^{-}]-f(t,u^{-})$, we compute
\begin{equation*}
\begin{split}
N[u^{-}]&=\ep e^{-\om(t-\tau)}[\om u+C\om(1-\ep e^{-\om(t-\tau)})u_{x}-f(t,u)+f_{u}(t,u_{*})u]\\
&\leq\ep e^{-\om(t-\tau)}[\om u+C\om(1-\ep_{*})u_{x}+f_{u}(t,u_{*})u],
\end{split}
\end{equation*}
where $u=u(t,x-C\ep e^{-\om(t-\tau)})$, $u_{x}=u_{x}(t,x-C\ep e^{-\om(t-\tau)})$ and $u_{*}\in[u^{-},u]$.

By Theorem \ref{thm-tf}$\rm(v)$, we find some $M_{0}>0$ such that
\begin{equation*}
\frac{u_{x}(t,x)}{u(t,x)}\leq-\frac{\ka}{2},\quad x\geq X(t)+M_{0},\quad t\in\R.
\end{equation*}
Let $M=M(\ep_{*})\geq M_{0}$ be such that
\begin{equation*}
u(t,x)\geq\frac{\theta_{1}}{1-\ep_{*}},\quad x\leq X(t)-M,\quad t\in\R.
\end{equation*}
By Lemma \ref{lem-uniform-steepness}, we have $C_{M}:=\sup_{t\in\R}\sup_{|x-X(t)|\leq M}u_{x}(t,x)<0$. Then, we choose $C=C(\ep_{*},\om)>0$ such that
\begin{equation*}
\begin{cases}
\om+C\om(1-\ep_{*})C_{M}+\sup_{(t,u)\in\R\times[0,1]}|f_{u}(t,u)|\leq0,\\
\om-\frac{C\om(1-\ep_{*})\ka}{2}+\sup_{(t,u)\in\R\times[0,1]}|f_{u}(t,u)|\leq0.
\end{cases}
\end{equation*}

Now, we consider three cases. If $x-C\ep e^{-\om(t-\tau)}-X(t)\leq-M$, we have $u\geq\frac{\theta_{1}}{1-\ep_{*}}$, and then, $u^{-}\geq(1-\ep e^{-\om(t-\tau)})\frac{\theta_{1}}{1-\ep_{*}}\geq\theta_{1}$. Thus, $f_{u}(t,u_{*})\leq-\beta_{1}$. Therefore, $\om\leq\beta_{1}$ leads to
\begin{equation*}
N[u^{-}]\leq\ep e^{-\om(t-\tau)}(\om+f_{u}(t,u_{*}))u\leq0.
\end{equation*}

If $x-C\ep e^{-\om(t-\tau)}-X(t)\in[-M,M]$, we have $u_{x}\leq C_{M}<0$. Hence, the choice of $C$ gives
\begin{equation*}
N[u^{-}]\leq\ep e^{-\om(t-\tau)}\Big[\om+C\om(1-\ep_{*})C_{M}+\sup_{(t,u)\in\R\times[0,1]}|f_{u}(t,u)|\Big]\leq0.
\end{equation*}

If $x-C\ep e^{-\om(t-\tau)}-X(t)\geq M$, we have $\frac{u_{x}}{u}\leq-\frac{\ka}{2}$. Then, the choice of $C$ leads to
\begin{equation*}
\begin{split}
N[u^{-}]&\leq \ep e^{-\om(t-\tau)}u\bigg[\om +C\om(1-\ep_{*})\frac{u_{x}}{u}+f_{u}(t,u_{*})\bigg]\\
&\leq\ep e^{-\om(t-\tau)}u\bigg[\om-\frac{C\om(1-\ep_{*})\ka}{2}+f_{u}(t,u_{*})\bigg]\leq0.
\end{split}
\end{equation*}

Hence, we have show $N[u^{-}]\leq0$, that is, $u^{-}$ is a sub-solution.
\end{proof}


\subsection{Proof of Theorem \ref{thm-asympt-stability}} This whole subsection is devoted to the proof of Theorem \ref{thm-asympt-stability}. Let $u_{0}$ be as in the statement of Theorem \ref{thm-asympt-stability} and $u(t,x;t_{0}):=u(t,x;t_{0},u_{0})$ be the solution of \eqref{main-eqn} with initial data $u(t_{0},\cdot;t_{0})=u_{0}$. We are going to show $\lim_{t\to\infty}\sup_{x\in\R}\big|\frac{u(t,x;t_{0})}{u(t,x)}-1\big|=0$, which is true if we can show
\begin{equation}\label{estimate-limsup}
\limsup_{t\to\infty}\sup_{x\in\R}\frac{u(t,x;t_{0})}{u(t,x)}\leq1
\end{equation}
and
\begin{equation}\label{estimate-liminf}
\liminf_{t\to\infty}\inf_{x\in\R}\frac{u(t,x;t_{0})}{u(t,x)}\geq1.
\end{equation}
Here, we only prove \eqref{estimate-limsup}; the proof of $\eqref{estimate-liminf}$ can be done along the same line except one that we comment after the proof of \eqref{estimate-limsup}.

To show \eqref{estimate-limsup}, we consider the set
\begin{equation*}
\Xi=\bigg\{\xi\geq0\bigg|\limsup_{t\to\infty}\sup_{x\in\R}\frac{u(t,x;t_{0})}{u(t,x-\xi)}\leq1\bigg\}.
\end{equation*}
We claim that $\Xi\neq\emptyset$. In fact, since $\lim_{x\to\infty}\frac{u_{0}(x)}{u(t_{0},x)}=1$, we have for any fixed $\ep_{*}\in(0,1-\theta_{1})$, $u_{0}(x)\leq(1+\ep_{*})u(t_{0},x)$ for all $x\gg1$. Then, we can find a large $\xi_{*}>0$ so that
\vspace{-0.04in}\begin{equation}\label{initial-estimate-95678}
u_{0}(x)\leq(1+\ep_{*})u(t_{0},x-\xi_{*}+C\ep_{*}), \quad x\in\R,
\vspace{-0.04in}\end{equation}
where $C=C(\ep_{*},\om_{*})$ is as in Proposition \ref{prop-sub-sup-solution}. Applying Proposition \ref{prop-sub-sup-solution}, we conclude
\vspace{-0.04in}\begin{equation*}
u(t,x;t_{0})\leq(1+\ep_{*}e^{-\om_{*}(t-t_{0})})u(t,x-\xi_{*}+C\ep_{*}e^{-\om_{*}(t-t_{0})}),\quad x\in\R
\vspace{-0.04in}\end{equation*}
for all $t\geq t_{0}$. This and the monotonicity of $u(t,x)$ yield $\limsup_{t\to\infty}\sup_{x\in\R}\frac{u(t,x;t_{0})}{u(t,x-\xi_{*})}\leq1$, that is, $\xi_{*}\in\Xi$.

Clearly, $\xi\in\Xi$ for all $\xi>\xi_{*}$. Set $\xi_{\inf}:=\inf\{\xi|\xi\in\Xi\}\geq0$. We claim that $\xi_{\inf}\in\Xi$. Indeed, let $\{\xi_{n}\}\subset\Xi$ and $\xi_{n}\to\xi_{\inf}$ as $n\to\infty$. We see
\begin{equation*}
\frac{u(t,x;t_{0})}{u(t,x-\xi_{\inf})}=\frac{u(t,x;t_{0})}{u(t,x-\xi_{n})}\bigg[1+\frac{u(t,x-\xi_{n})-u(t,x-\xi_{\inf})}{u(t,x-\xi_{\inf})}\bigg],
\end{equation*}
\begin{equation*}
\frac{u(t,x-\xi_{n})-u(t,x-\xi_{\inf})}{u(t,x-\xi_{\inf})}=\frac{u_{x}(t,x-\xi_{n}^{*})}{u(t,x-\xi_{\inf})}(\xi_{n}-\xi_{\inf}),\,\,\text{where}\,\,\xi_{n}^{*}\in[\xi_{\inf},\xi_{n}]
\end{equation*}
and
\begin{equation*}
\frac{u_{x}(t,x-\xi_{n}^{*})}{u(t,x-\xi_{\inf})}\leq\bigg[\sup_{(t,x)\in\R\times\R}\frac{|u_{x}(t,x)|}{u(t,x)}\bigg]\frac{u(t,x-\xi_{n}^{*})}{u(t,x-\xi_{\inf})}\leq\bigg[\sup_{(t,x)\in\R\times\R}\frac{|u_{x}(t,x)|}{u(t,x)}\bigg]\frac{e^{\ka|\xi_{n}^{*}-\xi_{\inf}|}}{h(1)},
\end{equation*}
where we used \eqref{harnack-type-tf}. Setting $C_{0}:=\frac{1}{h(1)}\sup_{(t,x)\in\R\times\R}\frac{|u_{x}(t,x)|}{u(t,x)}$, we then have
\begin{equation*}
\frac{u(t,x;t_{0})}{u(t,x-\xi_{\inf})}\leq\frac{u(t,x;t_{0})}{u(t,x-\xi_{n})}\big[1+C_{0}e^{\ka|\xi_{n}-\xi_{\inf}|}(\xi_{n}-\xi_{\inf})\big],
\end{equation*}
which leads to $\limsup_{t\to\infty}\sup_{x\in\R}\frac{u(t,x;t_{0})}{u(t,x-\xi_{\inf})}\leq1$, that is, $\xi_{\inf}\in\Xi$.

It remains to show that
\begin{equation}\label{shift-is-zero}
\xi_{\inf}=0.
\end{equation}
To do so, we need the following crucial result concerning the exact asymptotic of $u(t,x;t_{0})$ as $x\to\infty$.

\begin{lem}\label{lem-asymptotic-at-infty}
For any $\eta>0$, there exists $M=M(\eta)\gg1$ such that
\begin{equation}\label{sharp-asymptotics-key}
u(t,x+X(t)+\eta)\leq u(t,x+X(t);t_{0})\leq u(t,x+X(t)-\eta),\quad x\geq M
\end{equation}
for all $t\geq t_{0}$. Equivalently,
\begin{equation}\label{sharp-asymptotics-key-equiv}
\lim_{x\to\infty}\frac{u(t,x+X(t);t_{0})}{u(t,x+X(t))}=1\quad\text{uniformly in}\quad t\geq t_{0}
\end{equation}
\end{lem}
To show \eqref{estimate-limsup}, we actually only need the upper bound  in \eqref{sharp-asymptotics-key}; the lower bound in \eqref{sharp-asymptotics-key} and the limit \eqref{sharp-asymptotics-key-equiv} are needed for proving \eqref{estimate-liminf}. The proof of Lemma \ref{lem-asymptotic-at-infty} is postponed to Subsection \ref{subsec-proof-lemm-1}

Now, for contradiction, suppose \eqref{shift-is-zero} is false, that is, $\xi_{\inf}>0$. We are going to find a number in $\Xi$, but smaller than $\xi_{\inf}$.

To this end, let $\ep\in(0,\ep_{*})$ be sufficiently small, where $\ep_{*}\in(0,1-\theta_{1})$ is as in Proposition \ref{prop-sub-sup-solution}. Since $\xi_{\inf}\in\Xi$, i.e., $\limsup_{t\to\infty}\sup_{x\in\R}\frac{u(t,x;t_{0})}{u(t,x-\xi_{\inf})}\leq1$, for any $\ep_{1}>0$, we can find some $T=T(\ep_{1})$ such that $\frac{u(T,x;t_{0})}{u(T,x-\xi_{\inf})}<1+\ep_{1}$ for all $x\in\R$, which implies
\begin{equation*}\label{initial-condition-84294284284240539}
u(T,x;t_{0})< u(T,x-\xi_{\inf})+\ep_{1},\quad x\in\R.
\end{equation*}
In particular,
\begin{equation}\label{estimate-1359-1}
u(T,x+X(T);t_{0})<u(T,x+X(T)-\xi_{\inf})+\ep_{1},\quad x\in\R.
\end{equation}

Let $C=C(\ep_{*},\om_{*})$ be as in Proposition \ref{prop-sub-sup-solution}. We claim there exists $M_{1}>0$ such that
\begin{equation}\label{estimate-1359-2}
u(t,x+X(t))\leq (1+\ep)u(t,x+X(t)+3\ep C),\quad x\leq -M_{1},\quad t\in\R.
\end{equation}
Indeed, we have from Taylor expansion
\begin{equation*}
\begin{split}
&(1+\ep)u(t,x+X(t)+3\ep C)-u(t,x+X(t))\\
&\quad\quad=3\ep Cu_{x}(t,x+X(t)+x_{*})+\ep u(t,x+X(t)+3\ep C)\geq0,
\end{split}
\end{equation*}
where we used $u_{x}(t,x+X(t)+x_{*})\to0$ uniformly in $t\in\R$ as $x\to-\infty$ by Theorem \ref{thm-tf}$\rm(iv)$ and $u(t,x+X(t)+3\ep C)\to1$ uniformly in $t\in\R$ as $x\to-\infty$.

Now, applying \eqref{estimate-1359-2} with $t=T$, we deduce from \eqref{estimate-1359-1} that
\begin{equation}\label{estimate-1359-2-1}
u(T,x+X(T);t_{0})\leq(1+\ep)u(T,x+X(T)-\xi_{\inf}+3\ep C)+\ep_{1},\quad x\leq-M_{1}.
\end{equation}
Note that the inequality in \eqref{estimate-1359-1} is strict, which implies that if we choose $\ep$ so small that $3\ep C\ll1$, we will have
\begin{equation*}
u(T,x+X(T);t_{0})\leq u(T,x+X(T)-\xi_{\inf}+3\ep C)+\ep_{1},\quad x\in[-M_{1},M],
\end{equation*}
where $M=M(\frac{\xi_{\inf}}{2})\gg1$ is as in Lemma \ref{lem-asymptotic-at-infty}. This, together with \eqref{estimate-1359-2-1}, give
\begin{equation}\label{estimate-1359-3}
u(T,x+X(T);t_{0})\leq(1+\ep)u(T,x+X(T)-\xi_{\inf}+3\ep C)+\ep_{1},\quad x\leq M.
\end{equation}
Since $\inf_{t\in\R}\inf_{x\leq M}u(t,x+X(t)-\xi_{\inf}+1)>0$, we can take $\ep_{1}$ so small that
\begin{equation*}
\ep_{1}\leq\ep\inf_{t\in\R}\inf_{x\leq M}u(t,x+X(t)-\xi_{\inf}+1),
\end{equation*}
and then conclude from \eqref{estimate-1359-3} that
\begin{equation}\label{estimate-1359-4}
u(T,x+X(T);t_{0})\leq(1+2\ep)u(T,x+X(T)-\xi_{\inf}+3\ep C),\quad x\leq M.
\end{equation}

Using Lemma \ref{lem-asymptotic-at-infty}, we see that
\begin{equation}\label{estimate-1359-5}
u(T,x+X(T);t_{0})\leq u(T,x+X(T)-\frac{\xi_{\inf}}{2}),\quad x\geq M.
\end{equation}
If $\ep$ is so small that $3\ep C\leq\frac{\xi_{\inf}}{2}$, we have from \eqref{estimate-1359-5} that
\begin{equation*}
u(T,x+X(T);t_{0})\leq u(T,x+X(T)-\xi_{\inf}+3\ep C),\quad x\geq M.
\end{equation*}
This, together with \eqref{estimate-1359-4}, give
\begin{equation}\label{estimate-1359-6}
\begin{split}
u(T,x+X(T);t_{0})&\leq(1+2\ep)u(T,x+X(T)-\xi_{\inf}+3\ep C)\\
&=(1+2\ep)u(T,x+X(T)-\xi_{\inf}+\ep C+2\ep C),\quad x\in\R.
\end{split}
\end{equation}
We now apply Proposition \ref{prop-sub-sup-solution} to \eqref{estimate-1359-6} to conclude that
\begin{equation*}
\begin{split}
u(t,x+X(t);t_{0})&\leq(1+2\ep e^{-\om_{*}(t-T)})\\
&\quad\times u(t,x+X(t)-\xi_{\inf}+\ep C+2\ep Ce^{-\om_{*}(t-T)}),\quad x\in\R
\end{split}
\end{equation*}
for all $t\geq T$. From which and the monotonicity of $u(t,x)$ in $x$, we deduce
\begin{equation*}
\frac{u(t,x+X(t);t_{0})}{u(t,x+X(t)-\xi_{\inf}+\ep C)}\leq(1+2\ep e^{-\om_{*}(t-T)}),\quad x\in\R
\end{equation*}
for all $t\geq T$, which leads to $\limsup_{t\to\infty}\sup_{x\in\R}\frac{u(t,x;t_{0})}{u(t,x-\xi_{\inf}+\ep C)}\leq1$, that is, $\xi_{\inf}-\ep C\in\Xi$. It contradicts to the minimality of $\xi_{\inf}$. Hence, \eqref{shift-is-zero} holds, i.e., $\xi_{\inf}=0$.

Now, we comment on the proof of \eqref{estimate-liminf}. Define
\begin{equation*}
\tilde{\Xi}=\bigg\{\xi\geq0\bigg|\liminf_{t\to\infty}\inf_{x\in\R}\frac{u(t,x;t_{0})}{u(t,x-\xi)}\geq1\bigg\}.
\end{equation*}
To show $\tilde{\Xi}\neq\emptyset$, as in \eqref{initial-estimate-95678}, it seems that we have
\begin{equation*}
(1-\ep_{*})u(t_{0},x+\xi_{*}-C\ep_{*})\leq u_{0}(x),\quad x\in\R
\end{equation*}
for some $\ep_{*}\in(0,1-\theta_{1})$ and $\xi_{*}<0$, but it is wrong if $\limsup_{x\to-\infty}u_{0}(x)<\theta_{1}$. To overcome this, we only need to wait for a period of time, say $\tilde{T}=\tilde{T}(u_{0})>0$, so that $\liminf_{x\to-\infty}u(t_{0}+\tilde{T},x;t_{0})>\theta_{1}$. Then, at $t_{0}+\tilde{T}$, we still have limit $\lim_{x\to\infty}\frac{u(t_{0}+\tilde{T},x;t_{0})}{u(t_{0}+\tilde{T},x)}=1$ due to Lemma \ref{lem-asymptotic-at-infty}, which ensures
 \begin{equation*}
(1-\ep_{*})u(t_{0}+\tilde{T},x+\xi_{*}-C\ep_{*})\leq u(t_{0}+\tilde{T},x;t_{0}), \quad x\in\R
\end{equation*}
for some $\ep_{*}\in(0,1-\theta_{1})$ and $\xi_{*}<0$. Applying Proposition \ref{prop-sub-sup-solution}, we then see that $\tilde{\Xi}\neq\emptyset$. The rest of the proof follows along the same line. This completes the proof.


\subsection{Proof of Lemma \ref{lem-asymptotic-at-infty}}\label{subsec-proof-lemm-1}

Recall $u(t,x)=u^{\ka}(t,x)$ and $X(t)=c^{\ka}(t)$ for some fixed $\ka\in(0,\ka_{0}]$. Fix $\ka_{*}\in(\ka_{0},\inf_{t\in\R}\ka_{0}(t)]$, where $\ka_{0}(t)$ is given in \eqref{minimizer-148}. We prove the lemma within two steps.
\medskip

\noindent {\textbf{Step 1.}} We prove \eqref{sharp-asymptotics-key}. Fix any $\eta>0$. We first prove the upper bound for $u(t,x;t_{0})$. Since $\lim_{x\to\infty}\frac{u_{0}(x)}{u(t_{0},x)}=1$ by assumption and $\lim_{x\to\infty}\frac{u(t_{0},x+X(t_{0}))}{e^{-\ka x}}=1$ by Theorem \ref{thm-tf}$\rm(ii)$, we can find some $\de=\de(\eta)\gg1$ such that
\begin{equation*}
u_{0}(x+\eta)\leq e^{-\ka(x-X(t_{0}))}+\de e^{-\ka_{*}(x-X(t_{0}))},\quad x\in\R.
\end{equation*}
Recall that $e^{-\ka(x-X(t))}$ is a solution of $u_{t}=J\ast u-u+a(t)u$. Setting
\begin{equation*}
v(t,x)=e^{-\ka(x-X(t))}+\de e^{-\ka_{*}(x-X(t))},\quad x\in\R,\quad t\geq t_{0},
\end{equation*}
we readily check that
\begin{equation*}
v_{t}-[J\ast v-v]-a(t)v=\ka_{*}\bigg[\frac{\int_{\R}J(y)e^{\ka y}dy-1+a(t)}{\ka}-\frac{\int_{\R}J(y)e^{\ka_{*}y}dy-1+a(t)}{\ka_{*}}\bigg].
\end{equation*}
Since the function $\zeta\mapsto\frac{\int_{\R}J(y)e^{\zeta y}dy-1+a(t)}{\zeta}$ is convex in $\zeta$ and decreasing on $(0,\ka_{0}(t))$, and $\ka\leq\ka_{0}<\ka_{*}\leq\inf_{t\in\R}\ka_{0}(t)$, we have
\begin{equation*}
\frac{\int_{\R}J(y)e^{\ka y}dy-1+a(t)}{\ka}\geq\frac{\int_{\R}J(y)e^{\ka_{*}y}dy-1+a(t)}{\ka_{*}},\quad t\in\R.
\end{equation*}
Hence, $v_{t}-[J\ast v-v]-a(t)v\geq0$. In particular, $v(t,x)$ is a super-solution of \eqref{main-eqn}. It then follows from comparison principle and the space homogeneity of \eqref{main-eqn} that $u(t,x+\eta;t_{0})\leq v(t,x)$ for all $x\in\R$ and $t\geq t_{0}$, that is,
\begin{equation*}
\frac{u(t,x+X(t);t_{0})}{e^{-\ka(x-\eta)}}\leq1+\de e^{-(\ka_{*}-\ka)(x-\eta)},\quad x\in\R
\end{equation*}
for $t\geq t_{0}$,
which, together with $\lim_{x\to\infty}\frac{u(t,x+X(t))}{e^{-\ka x}}=1$ by Theorem \ref{thm-tf}$\rm(ii)$, lead to
\begin{equation*}
\frac{u(t,x+X(t);t_{0})}{u(t,x+X(t)-\eta)}\leq(1+\de e^{-(\ka_{*}-\ka)(x-\eta)})(1+\zeta(x-\eta)),\quad x\in\R
\end{equation*}
for $t\geq t_{0}$, where $\zeta(x)\geq0$ satisfies $\zeta(x)\to0$ as $x\to\infty$. It then follows that
\begin{equation*}
\frac{u(t,x+X(t);t_{0})}{u(t,x+X(t)-2\eta)}\leq\frac{u(t,x+X(t)-\eta)}{u(t,x+X(t)-2\eta)}(1+\de e^{-(\ka_{*}-\ka)(x-\eta)})(1+\zeta(x-y)).
\end{equation*}
Since
\begin{equation*}
\begin{split}
\frac{u(t,x+X(t)-\eta)}{u(t,x+X(t)-2\eta)}&=\frac{u(t,x+X(t)-\eta)}{e^{-\ka(x-\eta)}}\frac{e^{-\ka(x-\eta)}}{e^{-\ka(x-2\eta)}}\frac{e^{-\ka(x-2\eta)}}{u(t,x+X(t)-2\eta)}\\
&\to e^{-\ka\eta}<1\quad\text{as}\quad x\to\infty,
\end{split}
\end{equation*}
and $\eta>0$ is arbitrary, the upper bound for $u(t,x;t_{0})$ follows.

We now prove the lower bound for $u(t,x;t_{0})$. Since $\lim_{x\to\infty}\frac{u_{0}(x)}{u(t_{0},x)}=1$ by assumption and $\lim_{x\to\infty}\frac{u(t_{0},x)}{h(e^{-\ka(x-X(t_{0}))})}=1$ by \eqref{asymptotic-9564}, we can find some $\de=\de(\ep)\gg1$ such that
\begin{equation*}
u_{0}(x-\eta)\geq h(e^{-\ka(x-X(t_{0}))})-\de e^{-\ka_{*}(x-X(t_{0}))},\quad x\in\R.
\end{equation*}
Setting
\begin{equation*}
v_{1}(t,x)=h(e^{-\ka(x-X(t))})\quad\text{and}\quad v_{2}(t,x)=e^{-\ka_{*}(x-X(t))}
\end{equation*}
for $x\in\R$ and $t\geq t_{0}$. By Proposition \ref{prop-subsol}, $v_{1}(t,x)$ is a sub-solution of $u_{t}=J\ast u-u+a(t)g(u)$, i.e., $(v_{1})_{t}\leq J\ast v_{1}-v_{1}+a(t)g(v_{1})$. Arguing as in the proof of the upper bound above, we see that
\begin{equation*}
\begin{split}
&(v_{2})_{t}-[J\ast v_{2}-v_{2}]-a(t)v_{2}\\
&\quad\quad=\ka_{*}\bigg[\frac{\int_{\R}J(y)e^{\ka y}dy-1+a(t)}{\ka}-\frac{\int_{\R}J(y)e^{\ka_{*}y}dy-1+a(t)}{\ka_{*}}\bigg]\geq0.
\end{split}
\end{equation*}
It then follows that $v(t,x)=v_{1}(t,x)-\de v_{2}(t,x)$ satisfies
\begin{equation*}
\begin{split}
v_{t}-[J\ast v-v]&=\big((v_{1})_{t}-[J\ast v_{1}-v_{1}]\big)-\de\big((v_{2})_{t}-[J\ast v_{2}-v_{2}]\big)\\
&\leq a(t)g(v_{1})-\de a(t)v_{2}\leq a(t)g(v),
\end{split}
\end{equation*}
where we used $g(v_{1})-g(v)=\de v_{2}g_{u}(v_{*})\leq\de v_{2}$, since $g_{u}\leq1$ (note that it's safe to extend $g$ to $(-\infty,0)$ so that $g_{u}\leq1$ on this interval). Since $g(u)\leq f(t,u)$, we find $v_{t}\leq J\ast v-v+f(t,u)$. Then, by comparison principle, $u(t,x-\eta;t_{0})\geq v(t,x)$ for all $x\in\R$ and $t\geq t_{0}$, that is,
\begin{equation*}
\frac{u(t,x+X(t);t_{0})}{h(e^{-\ka(x+\eta)})}\geq1-\de\frac{e^{-\ka_{*}(x+\ep)}}{h(e^{-\ka(x+\eta)})},\quad x\in\R
\end{equation*}
for $t\geq t_{0}$, which, together with \eqref{asymptotic-9564}, lead to
\begin{equation*}
\frac{u(t,x+X(t);t_{0})}{u(t,x+X(t)+\eta)}\geq\bigg(1-\de\frac{e^{-\ka_{*}(x+\eta)}}{h(e^{-\ka(x+\eta)})}\bigg)(1-\zeta(x+\eta)),\quad x\in\R
\end{equation*}
for all $t\geq t_{0}$, where $\zeta(x)$ satisfies $\zeta(x)\to0$ as $x\to\infty$. Since $\ka_{*}>\ka$ and $\lim_{x\to\infty}\frac{e^{-\ka x}}{h(e^{-\ka x})}=1$, we have $\lim_{x\to\infty}\frac{e^{-\ka_{*}x}}{h(e^{-\ka x})}=0$. Then, a similar argument as in the proof of the upper bound gives the lower bound for $u(t,x;t_{0})$. This finishes the proof of \eqref{sharp-asymptotics-key}.

\medskip

\noindent {\textbf{Step 2.}} We prove the equivalence between \eqref{sharp-asymptotics-key} and \eqref{sharp-asymptotics-key-equiv}.
Suppose first that \eqref{sharp-asymptotics-key} holds. Then,
\begin{equation*}
\frac{u(t,x+X(t)+\eta)}{u(t,x+X(t))}\leq\frac{u(t,x+X(t);t_{0})}{u(t,x+X(t))}\leq\frac{u(t,x+X(t)-\eta)}{u(t,x+X(t))},\,\, x\geq M,\,\, t\geq t_{0}.
\end{equation*}
We see that
\begin{equation*}
\begin{split}
\frac{u(t,x+X(t)-\eta)}{u(t,x+X(t))}&=1-\frac{u_{x}(t,x+X(t)-\eta_{*})}{u(t,x+X(t)-\eta_{*})}\frac{u(t,x+X(t)-\eta_{*})}{u(t,x+X(t))}\eta\\
&\leq1+\bigg[\sup_{(t,x)\in\R\times\R}\frac{|u_{x}(t,x)|}{u(t,x)}\bigg]\frac{e^{\ka|\eta_{*}|}}{h(1)}\eta,
\end{split}
\end{equation*}
where $\eta_{*}\in[0,\eta]$ and we used \eqref{harnack-type-tf}. Similarly, we argue
\begin{equation*}
\frac{u(t,x+X(t)+\eta)}{u(t,x+X(t))}\geq1-\bigg[\sup_{(t,x)\in\R\times\R}\frac{|u_{x}(t,x)|}{u(t,x)}\bigg]\frac{e^{\ka|\eta_{*}|}}{h(1)}\eta.
\end{equation*}
Setting $\tilde{C}=\frac{1}{h(1)}\big[\sup_{(t,x)\in\R\times\R}\frac{|u_{x}(t,x)|}{u(t,x)}\big]$, we obtain
\begin{equation*}
1-\tilde{C}e^{\ka|\eta_{*}|}\eta\leq\frac{u(t,x+X(t);t_{0})}{u(t,x+X(t))}\leq1+\tilde{C}e^{\ka|\eta_{*}|}\eta,\quad x\geq M,\quad t\geq t_{0}.
\end{equation*}
Setting $x\to\infty$, we find
\begin{equation*}
1-\tilde{C}e^{\ka|\eta_{*}|}\eta\leq\liminf_{x\to\infty}\frac{u(t,x+X(t);t_{0})}{u(t,x+X(t))}\leq\limsup_{x\to\infty}\frac{u(t,x+X(t);t_{0})}{u(t,x+X(t))}\leq1+\tilde{C}e^{\ka|\eta_{*}|}\eta,\quad t\geq t_{0},
\end{equation*}
which leads to \eqref{sharp-asymptotics-key-equiv} since $\eta>0$ is arbitrary.

Conversely, suppose \eqref{sharp-asymptotics-key-equiv} holds. We only show the upper bound for $u(t,x;t_{0})$ in \eqref{sharp-asymptotics-key}; the lower bound can be verified similarly. By \eqref{sharp-asymptotics-key-equiv}, for any small $\eta>0$, there exists $M(\eta)>0$ such that
\begin{equation*}
\frac{u(t,x+X(t);t_{0})}{u(t,x+X(t))}\leq\frac{1}{1-\frac{\ka}{8}\eta e^{-\ka\eta}},\quad x\geq M(\eta),\quad t\geq t_{0}.
\end{equation*}
We claim that
\begin{equation}\label{claim-equiv}
\begin{cases}
\exists M_{1}(\eta)>0\,\,\text{s.t.}\,\, u(t,x+X(t))\leq(1-\frac{\ka}{8}\eta e^{-\ka\eta})u(t,x+X(t)-\eta)\\
\quad\quad \text{for all}\,\, x\geq M_{1}(\eta)\,\,\text{and}\,\, t\geq t_{0}.
\end{cases}
\end{equation}
Clearly, the result then follows from \eqref{claim-equiv}. Hence, it remains to show \eqref{claim-equiv}. We see that
\begin{equation*}
\begin{split}
&\frac{u(t,x+X(t))}{u(t,x+X(t)-\eta)}\\
&\quad\quad=1-\bigg[\frac{-u_{x}(t,x+X(t)-\eta_{*})}{u(t,x+X(t)-\eta_{*})}\frac{u(t,x+X(t)-\eta_{*})}{e^{-\ka(x-\eta_{*})}}\frac{e^{-\ka(x-\eta_{*})}}{e^{-\ka(x-\eta)}}\frac{e^{-\ka(x-\eta)}}{u(t,x+X(t)-\eta)}\eta\bigg],
\end{split}
\end{equation*}
where $\eta_{*}\in[0,\eta]$. Clearly, $\frac{e^{-\ka(x-\eta_{*})}}{e^{-\ka(x-\eta)}}=e^{-\ka(\eta-\eta_{*})}\geq e^{-\ka\eta}$. Moreover, by Theorem \ref{thm-tf}, we can find some $M_{1}(\eta)>0$ such that for all $x\geq M_{1}(\eta)$ and $t\geq t_{0}$, there hold the following
\begin{equation*}
\frac{-u_{x}(t,x+X(t)-\eta_{*})}{u(t,x+X(t)-\eta_{*})}\geq\frac{\ka}{2},\quad\frac{u(t,x+X(t)-\eta_{*})}{e^{-\ka(x-\eta_{*})}}\geq\frac{1}{2}\quad\text{and}\quad\frac{e^{-\ka(x-\eta)}}{u(t,x+X(t)-\eta)}\geq\frac{1}{2}.
\end{equation*}
It follows that $\frac{u(t,x+X(t))}{u(t,x+X(t)-\eta)}\leq1-\frac{\ka}{8}\eta e^{-\ka\eta}$ for all $x\geq M_{1}(\eta)$ and $t\geq t_{0}$, that is, \eqref{claim-equiv} is true, and the proof is complete.


\subsection{Stability: general cases}\label{subsec-stability-general}

We prove the following generalizations of Theorem \ref{thm-asympt-stability}.

\begin{cor}\label{cor-stability-general}
Assume (H1)-(H3). Let $\ka_{0}$ be as in Theorem \ref{thm-tf}. For $\ka\in(0,\ka_{0}]$, let $u^{\ka}(t,x)$ be the transition front in Theorem \ref{thm-tf}. Let $u_{0}:\R\to[0,1]$ be uniformly continuous and satisfy $\liminf_{x\to-\infty}u_{0}>0$ and $\lim_{x\to\infty}\frac{u_{0}(x)}{u^{\ka}(t_{0},x)}=\la$ for some $t_{0}\in\R$ and $\la\in(0,\infty)$. Then, there holds the limit
\begin{equation*}
\lim_{t\to\infty}\sup_{x\in\R}\bigg|\frac{u(t,x;t_{0},u_{0})}{u^{\ka}(t,x-\frac{\ln\la}{\ka})}-1\bigg|=0.
\end{equation*}
\end{cor}
\begin{proof}
Write $u(t,x)=u^{\ka}(t,x)$. To apply Theorem \ref{thm-asympt-stability}, we only need to show $\lim_{x\to\infty}\frac{u_{0}(x)}{u(t_{0},x-\frac{\ln\la}{\ka})}=1$, which follows from
\begin{equation*}
\frac{u_{0}(x)}{u(t_{0},x-\frac{\ln\la}{\ka})}=\frac{u_{0}(x)}{u(t_{0},x)}\frac{u(t_{0},x)}{e^{-\ka(x-X(t_{0}))}}\frac{e^{-\ka(x-X(t_{0}))}}{e^{-\ka(x-X(t_{0})-\frac{\ln\la}{\ka})}}\frac{e^{-\ka(x-X(t_{0})-\frac{\ln\la}{\ka})}}{u(t_{0},x-\frac{\ln\la}{\ka})},
\end{equation*}
and
\begin{equation*}
\lim_{x\to\infty}\frac{u(t_{0},x)}{e^{-\ka(x-X(t_{0}))}}=1=\lim_{x\to\infty}\frac{e^{-\ka(x-X(t_{0})-\frac{\ln\la}{\ka})}}{u(t_{0},x-\frac{\ln\la}{\ka})}
\end{equation*}
thanks to Theorem \ref{thm-tf}(ii).
\end{proof}

\begin{cor}\label{cor-stability-more-general}
Assume (H1)-(H3). Let $\ka_{0}$ be as in Theorem \ref{thm-tf}. For $\ka\in(0,\ka_{0}]$, let $u^{\ka}(t,x)$ be the transition front in Theorem \ref{thm-tf}. Let $u_{0}:\R\to[0,1]$ be uniformly continuous and satisfy $\liminf_{x\to-\infty}u_{0}>0$ and $\lim_{x\to\infty}\frac{u_{0}(x)}{e^{-\ka x}}=\la$ for some $\la>0$. Then, there exists $t_{0}\in\R$ such that
\begin{equation*}
\lim_{t\to\infty}\sup_{x\in\R}\bigg|\frac{u(t,x;t_{0},u_{0})}{u^{\ka}(t,x)}-1\bigg|=0.
\end{equation*}
\end{cor}
\begin{proof}
Write $u(t,x)=u^{\ka}(t,x)$. By assumption, we see that for any $t\in\R$, there holds the limit $\lim_{x\to\infty}\frac{u_{0}(x)}{e^{-\ka(x-X(t))}}=\la e^{-\ka X(t)}$. Then, by Theorem \ref{thm-tf}(ii), for any $t\in\R$, we have the limit $\lim_{x\to\infty}\frac{u_{0}(x)}{u(t,x)}=\la e^{-\ka X(t)}$. Since $X(t)$ is continuous, increasing and satisfies $\lim_{t\to\pm\infty}X(t)=\pm\infty$, there exists a unique $t_{0}\in\R$ such that $\la e^{-\ka X(t_{0})}=1$, which implies that $\lim_{x\to\infty}\frac{u_{0}(x)}{u(t_{0},x)}=1$. We then apply Theorem \ref{thm-asympt-stability} to conclude the result.
\end{proof}

\section*{Acknowledgements}
The authors would like to thank anonymous referees for carefully reading the manuscript and providing useful suggestions.


\bibliographystyle{amsplain}

\end{document}